\documentclass[12pt,a4paper]{article}

\usepackage[utf8]{inputenc}
\usepackage{amsmath,amsthm,amscd,amssymb,eucal,mathrsfs}
\usepackage{amsfonts}
\usepackage{tikz}
\usetikzlibrary{positioning}
\tikzset{main node/.style={circle,fill=blue!20,draw,minimum size=1cm,inner sep=0pt},
            }
\usepackage{amssymb}
\usepackage{amsfonts}
\usepackage{afterpage}
\usepackage{graphicx}
\graphicspath{ {c:/user/images/} }

\newtheorem{teo}{Theorem}[section]
\newtheorem{defi}[teo]{Definition}
\newtheorem{prop}[teo]{Proposition}
\newtheorem{exm}[teo]{Example}
\newtheorem{cor}[teo]{Corollary}
\newtheorem{lem}[teo]{Lemma}

\newtheorem{rem}[teo]{Remark}

\newcommand{\mathset}[1]{{\left\{#1\right\}}}
\newcommand{\absolute}[1]{\left\lvert#1\right\rvert}

\DeclareMathOperator{\Span}{Span}
\DeclareMathOperator{\Vol}{Vol}
\DeclareMathOperator{\Spec}{Spec}
\DeclareMathOperator{\diag}{diag}
\DeclareMathOperator{\supp}{supp}

\title{A non-autonomous $p$-Adic diffusion equation on time changing graphs}
\author{Patrick Erik Bradley and \'Angel Mor\'an Ledezma}
\date{\today}

\begin{document}

\maketitle

\begin{abstract}
Motivated by the recently proven presence of ultrametricity in physical models (certain spin glasses) and the very recent study of Turing patterns on locally ultrametric state spaces, first non-autonomous diffusion operators on such spaces, where finitely many compact $p$-adic spaces are joined by a graph structure, are studied, including their Dirichlet and van Neumann eigenvalues. Secondly, the Cauchy problem for the heat equation associated with these operators is solved, its solution approximated by Trotter products, and thirdly, the corresponding Feller property as well as the Markov property (a Hunt process) is established. 
\end{abstract}

\section{Introduction}

The Gibbs measure associated with certain mean-field spin glass models, including the Sherrington-Kirkpatrick model,  was conjectured to  have ultrametric support in the thermodynamic limit
\cite{MPSTV1984}. This conjecture, known as the \emph{Parisi ultrametricity conjecture}, was proved by Panchenko  under the condition that the so-called \emph{Ghirlanda-Guerra identities} are satisfied \cite{Panchenko2013}.
This conjecture (and now result) has motivated up to this day the study of random processes on ultrametric state spaces, the most studied type being  diffusion on the $p$-adic numbers, cf.\ e.g.\ \cite{Mystkowski1994,Karwowski2007} and many other research articles in the physical and mathematical literature.
\newline

Ultrametricity of a space means that it has a hierarchical structure, i.e.\ its points can be viewed as the boundary of a tree, and therefore its topology, as well as any physical process on it, are determined by this tree structure. The $p$-adic numbers form a convenient mathematical framework for studying ultrametric spaces by viewing those trees as part of the regular tree structure of the $p$-adic numbers---as long as the  vertex degrees do not exceed the bound given by the $p$-adic regularity.
However, if the state space is only locally ultrametric, then it seemed unclear how to mathematically model diffusion on such structures until Z\'u\~niga-Galindo's idea of modelling interactions between those local ultrametric pieces appeared via an externally imposed graph structure on this ensemble of pieces which then can be viewed as disjoint compact open subsets of the $p$-adic number field $\mathbb{Q}_p$ \cite{ZunigaNetworks}. He even argues that $p$-adic analysis is a natural tool for studying Turing patterns on mean-field models.
In his case, these local pieces were modelled as $p$-adic balls, and thus he was lead to define a $p$-adic Laplacian integral operator resembling the integral operator representation of the well-studied Vladimirov (or Vladimirov-Taibleson) operator, cf.\ \cite{Taibleson1975,VVZ1994}.
\newline

The study of non-autonomous diffusion on $p$-adic domains seems to have been given significantly less consideration in the literature.
This motivates the authors of this article to investigate diffusion processes on locally hierarchical structures whose interactions (given as weighted edges) are time-dependent. The vertices of the underlying time-dependent graph are fixed in this setting. After clarifying the construction of $p$-adic operators associated with finite graphs---the ``dictionary'' set up in \cite{brad_wave_p} needs a correction---properties of a time-dependent version of the Z\'u\~niga operators from \cite{ZunigaNetworks} are studied, proving that the uniform continuity of the kernel function  implies the strong continuity of the new time-dependent operator (Theorem \ref{stronglyContinuous}) which is necessary for proving the Feller property later. 
Dirichlet and van Neumann eigenvalues  are calculated (Theorem \ref{boundaryEV}), whereby the corresponding boundary conditions are generalised from graph theory. Thereafter,
the non-autonomous Cauchy problem for  the heat equation associated with these operators is solved (Theorem \ref{CauchyProblem}), thereby approximating the solution using Trotter products, and then finding that the temporal discretisation error is in first order $O\left(\frac{1}{n}e^{\frac{1}{n}}\right)$ (Corollary \ref{error}). In the end the Feller property (Theorem \ref{contractionSG}) and corresponding Markov property---a Hunt process---is established (Theorem \ref{Markov}).

\section{Non-autonomous $p$-adic equations for time-changing graphs}

\subsection{Preliminaries}

In order to fix notation, denote the field of $p$-adic numbers as $\mathbb{Q}_p$. The usual $p$-adic absolute value on $\mathbb{Q}_p$ is written as $\absolute{\cdot}_p$. The additive group of $\mathbb{Q}_p$ being a locally compact abelian group means that it has a Haar measure, denoted as $dx$ (if the current variable is $x$). It is invariant under translation by adding any fixed element of $\mathbb{Q}_p$, used for integrating functions, and is normalised here by setting
\[
\int_{\mathbb{Z}_p}dx=1
\]
where $\mathbb{Z}_p$ is the subring of $\mathbb{Q}_p$ consisting of elements with $p$-adic absolute value at most $1$, i.e.\ $\mathbb{Z}_p$ is the $p$-adic unit ball.
\newline

    For each $t\in [0,\infty)$, we take a matrix $A(t)\in \mathbb{C}^{n\times n}$. If for every $t\in [0,\infty)$ we assume that $A(t)$ is a symmetric matrix, then we can view this matrix as an adjacency matrix of a weighted graph $\mathcal{G}_t$ with $n$ nodes in a natural way. For every $t$, the vertex set of $\mathcal{G}_t$ can be embedded into $\mathbb{Q}_p$. The fact that $\mathcal{G}_t$ has an unchanging set of nodes allows to fix an embedding for all $t>0$. In other words the only difference between $\mathcal{G}_t$ and $\mathcal{G}_s$ for $s\neq t$ are the edge weights. The set of vertices being constant, allows to write: $V(\mathcal{G}):=V(\mathcal{G}_t)$, for any $t>0$. 
Following \cite{ZunigaNetworks},    
    we fix the embedding 
$$V(\mathcal{G})\rightarrow\mathbb{Q}_p,$$
and fix $N$ such that each vertex $I\in V(G)$ lies in a distinct $p$-adic ball of radius $r=p^{-N}$
\[
B_r(I)=\{x\in \mathbb{Q}_p: |x-I|\leq r\}
\]
sufficiently small such that all these balls are disjoint in $\mathbb{Q}_p$. Let 
\[
K_N =\bigsqcup_{I\in V(\mathcal{G})}B_{p^{-N}}(I). 
\]
Each vertex of $\mathcal{G}$ can be assumed to have the form 
\[
I=I_0+I_1 p+...+I_{N-1}p^{N-1},
\]
where each $I_i\in\{0,\dots,p-1\}$ for $i=0,1,2,\dots$ is a $p$-adic digit.

\subsection{Basis of $L^2(K_N,\mathbb{C})$}

We are interested in the Hilbert space $L^2(K_N,\mathbb{C})$ with inner product  
$$\langle f, g \rangle =\int_{K_N} f(x) \ \overline{g(x)}\,  dx$$
Let $\Omega(p^r|x-a|_p)$ denote the indicator function of the ball $B_{-r}(a)$. If $r=N$ and $a=I\in V(G)$, we write 
\[
\Omega(p^N|x-I|_p)=:\Omega_{I,N}(x)
\]
The space of functions $X_N=\Span_{\mathbb{C}}\{\Omega_{I,N}(x)\}\cong \mathbb{C}^n$ is an $n$-domensional vector space and it will be denoted by $X_N$. 

\begin{defi}
The functions 
$$\Psi_{j,n,r}(x)=p^{\frac{r}{2}}\chi_p(p^{-r-1}jx)\Omega(p^{r}|x-p^{-r}n|_p)$$
where $r\in \mathbb{Z}$, $j\in \{1,...,p-1\}$, and $n$ runs through a fixed set of representatives of $\mathbb{Q}_p / \mathbb{Z}_p$ are called \textbf{Kozyrev functions}.  
ion. If $Supp (\Psi_{j,n,r})=B_{-r}(p^{-r}n) \subset B_{-N}(I)$, then $|p^{-r}n|_p=|I|_p$, with $N\leq r$. 
\end{defi}
\textit{Proof}. We have that $p^{-r}\leq p^{-N}$, and 
$$|p^{-r}n-I|_p<p^{-N}\implies |p^{-N-r}n-p^{-N}I|_p<1,$$
where $p^{-N}I\in \mathbb{Q}_p/\mathbb{Z}_p$, and therefore $|p^{-N}I|_p>1$. Hence, if $p^{-N-r}n\neq p^{-N}I$, we have
\[
1>|p^{-N-r}n-p^{-N}I|_p= \max\{|p^{-N-r}n|_p,|p^{-N}I|_p\}>1.
\]
This shows that we must have $|p^{-N-r}n|_p=|p^{-N}I|_p$, implying the assertion. \qed \newline

Therefore an orthonormal basis of 
\[
\mathcal{L}_0^2(K_N)=\mathset{f\in L^2(K_N)\mid\int_{K_N}f(x)\,dx=0}
\]
is the set of functions of the form 
\begin{equation}\label{KozyrevWavelet}
    \Psi_{j,I,r}(x)=p^{\frac{r}{2}}\chi_p(p^{-r-1}jx)\Omega(p^{r}|x-I|_p),
\end{equation} with $r\geq N$, $j=0,...p-1$ and $I\in V(\mathcal{G})$.

\begin{prop}
We have $L^2(K_N,\mathbb{C})=X_N \oplus \mathcal{L}_0^2(K_N)$.
\end{prop}

\textit{Proof.} 
This is shown in \cite[(10.8)]{ZunigaNetworks}, where orthogonality uses the well-known result:
\begin{equation}\label{orthogonality}
    \langle \Omega_{I,N},\Psi_{j,J,r}\rangle=0,
\end{equation}
for every $r\geq N$, $j=0,...,p-1$ and $I,J\in V(\mathcal{G})$, cf.\ e.g.\ \cite[Thm.\ 3.9]{XKZ2018} or \cite[Thm.\ 9.4.2]{AXS2010}. \qed \newline

By the above proposition, the set of functions consisting of the functions $p^{-\frac{N}{2}}\Omega_{I,N}(x)$ and the Kozyrev functions of the form (\ref{KozyrevWavelet}) form an orthonormal basis of $L^2(K_N,\mathbb{C})$. 

\subsection{Two operators over $L^2(K_N,\mathbb{C})$}
For every matrix $A\in \mathbb{C}^{n\times n}$ we can define a bounded linear operator in the following way. We define 
\[
A(x,y)=p^{N}\sum_{I,J\in V(\mathcal{G})}A_{I,J}\Omega_{I,N}(x)\Omega_{J,N}(y). 
\]
This function is a test function on $K_N \times K_N$, and hence, $A(x,y)\in L^2(K_N\times K_N, \mathbb{C})$. This function gives rise to an integral operator in the usual way: 
\begin{equation}\label{integralOp}
    \mathcal{A} u(x)=\int_{K_N} A(x,y)u(y)dy.
\end{equation}
This operator is a bounded linear operator from $L^2(K_N,\mathbb{C})$ to itself. 
\begin{prop}\label{kernelAdjacency}
Let $A\in \mathbb{C}^{n\times n}$ and $\mathcal{A}$ its integral operator of the form (\ref{integralOp}). Then $X_N^{\perp}\subset Ker(\mathcal{A})$, where $X_N^{\perp}$ is the orthogonal complement of $X_N$.
\end{prop}  
\textit{Proof}. Note that 
$$\mathcal{A}u(x)=p^N\sum_{I,J\in V(\mathcal{G})} A_{I,J}\langle u,\Omega_{J,N}\rangle\Omega_{I,N}(x)\in X_N.$$
If $u\in X_N^{\perp}$, then $\langle u, \Omega_{I,N}\rangle=0$, therefore by the above equality $\mathcal{A}u=0$.   \qed 
\newline
\begin{cor}
For every matrix $A\in \mathbb{C}^{n\times n}$, and every $u\in \mathcal{L}_0^2(K_N)$, $\mathcal{A}u=0$. In particular the set Kozyrev functions of the form (\ref{KozyrevWavelet}) are contained in $Ker (\mathcal{A})$. 
\end{cor}
\textit{Proof.} By Proposition \ref{kernelAdjacency} and equality (\ref{orthogonality}), the result follows. \qed
\newline

Recall that the \textit{Laplacian matrix} $L$ is defined as 
$$L=A-(\gamma_{I}\delta_{IJ})$$
where $A$ is the adjacency matrix of a graph $\mathcal{G}$, $\gamma_I=\sum_{J}A_{I,J}$ is the \textit{degree} of the vertex $I$, and $\delta_{I,J}$ is the identity matrix.  
We can associate to $L$ an operator of the form (\ref{integralOp}) denoted by $\mathcal{L}$. By Corollary \ref{kernelAdjacency},
the operator $\mathcal{L}$ only "acts" on the space $X_N$, that is the operator $\mathcal{L}$
defined
similarly as in (\ref{integralOp}):
\[
\mathcal{L}u(x)=\int_{K_N}L(x,y)u(y)\,dy
\]
with
\[
L(x,y)=p^N\sum\limits_{I,J\in V(\mathcal{G})}(A_{IJ}-\gamma_I\delta_{IJ})\Omega_{I,N}(x)\Omega_{J,N}(x)
\]
is just a rewrite of the Laplacian matrix in a space of dimension $n$, namely $X_N\cong \mathbb{C}^n$ and does not give an outcome of functions in $\mathcal{L}_0^2(K_N)$. In particular, the Kozyrev functions of the form (\ref{KozyrevWavelet}) are not associated to any non zero eigenvalue of the operator $\mathcal{L}$. Therefore the operator $\mathcal{L}$ is not an adequate candidate for the $p$-adic Laplacian, because it gives the same information as the Laplacian matrix, but not more. 

\begin{rem}\label{wrongStatement}
This is important in order to define an analogous operator to the Laplacian operator in the $p$-adic case. In \cite[Lemma 2.7]{brad_wave_p}  it is stated that the operator $\mathcal{L}$ coincides with the operator defined by the integral
$$ \mathbb{L}u(x)=\int_{K_N}A(x,y) (u(y)-u(x))dy$$
in the space of test functions, but this is not true as the following computations shows: 
Note that
\begin{align*}
\mathbb{L}u(x)
&=\int_{K_N}A(x,y) (u(y)-u(x))dy
\\
&=p^{N}\sum_{I,J}A_{I,J}\langle u,\Omega_{J,N}\rangle\Omega_{I,N}(x)-u(x)\gamma(x),
\end{align*}
where
\[
\gamma(x)=\int_{K_N}A(x,y)dy=\sum_{I,J}A_{I,J}\Omega_{I,N}(x).
\]
Then let $f=\Omega(p^{r}|x-a|_p)$ be the indicator function of a ball $B_{-r}(a)\subset B_{-N}(J_0)$, for some $J_0\in V(\mathcal{G})$. Then 
\begin{equation*}
    \begin{split}
        \mathcal{L}f(x)&=p^{N}\sum_{I,J}L_{I,J}\langle f, \Omega_{J,N}\rangle \Omega_{I,N}(x)\\
        &=p^{N}\sum_{I,J}L_{I,J}p^{-r}\delta_{J_0,J}\Omega_{I,N}(x)\\
        &=p^{N-r}\sum_{I}L_{I,J_0}\Omega_{I,N}(x)\\
        &=p^{N-r}\sum_{I}A_{I,J_0}\Omega_{I,N}(x)-p^{N-r}\gamma_{J_0}\Omega_{J_0,N}(x),
    \end{split}
\end{equation*}
On the other hand, we have 
\begin{equation*}
    \begin{split}
        \mathbb{L}f(x)&=\int_{K_N}A(x,y) (f(y)-f(x))dy
        \\
        &=p^{N}\sum_{I,J}A_{I,J}\langle f(y),\Omega_{J,N}(y)\rangle\Omega_{I,N}(x)-f(x)\gamma(x),\\
        &=p^{N}\sum_{I,J}A_{I,J}p^{-r}\delta_{J_0,J}\Omega_{I,N}(x)-f(x)\gamma(x)\\
        &=p^{N-r}\sum_{I}A_{I,J_0}\Omega_{I,N}(x)-\Omega(p^{r}|x-a|_p)\gamma(x),
    \end{split}
\end{equation*}
If $B_{-r}(a)\subset B_{-N}(J_0)$ is a proper inclusion, then we can take an $x_0\in B_{-N}(J_0) \setminus B_{-r}(a)$. Then
\[
\mathcal{L}f(x_0)=-p^{N-r}\gamma_{J_0}, 
\]
and
\[
\mathbb{L}f(x_0)=0.
\]
Here we supposed that $A_{I,I}=0$ for all $I$, i.e.\ that $\mathcal{G}$ is a simple graph. It is clear that, if $r=N$ and $a=I$, then equality holds, that is, the two operators coincide in the space $X_N$. Moreover, since  $\mathcal{L}$ and $\mathcal{A}$ (the operator associated to the adjacency matrix of $\mathcal{G}$) are zero for all $u\in X_{N}^{\perp}$, we have for such $u$ that
\[
\mathcal{L}u(x)=0, 
\]
and
$$\mathbb{L}u(x)=\mathcal{A}u(x)-u(x)\gamma(x)=-u(x)\gamma(x).$$
There is another fundamental difference between these two operators, as the following results shows. 
\end{rem}

\begin{prop}
The operator $\mathcal{L}$ is compact.
\end{prop}

\textit{Proof}. It is clear that $\mathcal{L}$ is a Hilbert-Schmidt operator, since it is zero for all the Kozyrev functions supported in $K_N$. It is known that all Hilbert-Schmidt operators are compact. \qed

\begin{prop}\label{degreeEV}
The operator $\mathbb{L}:L^{2}(K_N,\mathbb{C})\rightarrow L^{2}(K_N,\mathbb{C})$ defined by
\[
\mathbb{L}u(x)=\mathcal{A}u(x)-u(x)\gamma(x),
\]
where 
\[
\gamma(x)=\int_{K_N}A(x,y)dy=\sum_{I\in V(\mathcal{G})}\gamma_I \Omega_{I,N}(x),
\]
is a bounded linear operator, and all the Kozyrev functions of the form (\ref{KozyrevWavelet}) are eigenfunctions.
\end{prop}

\textit{Proof.} 
In \cite[Lem.\ 10.1]{ZunigaNetworks} it is shown that $\mathbb{L}$ is compact. The eigenfunction property and corresponding eigenvalues of the Kozyrev wavelets are shown in \cite[Thm.\ 10.1]{ZunigaNetworks}.
Notice that the claim in \emph{loc.\ cit}.\ about
$\mathbb{L}$ being compact is false, cf.\ Corollary \ref{notCompact} below.
\qed

\begin{cor}\label{notCompact}
$\mathbb{L}$ is not compact.
\end{cor}

\begin{proof}
This is an immediate consequence of Proposition \ref{degreeEV}.
\end{proof}

Notice that the non-compactness of $\mathbb{L}$ was also observed in
\cite[Ex.\ 1]{brad_HeatMumf}.
\newline

\begin{defi}
The operator $\mathcal{A}$ is called the \emph{$p$-adic Z\'u\~niga operator} associated with the finite graph $\mathcal{G}$, and
$\mathbb{L}$  is called the corresponding \emph{$p$-adic Z\'u\~niga Laplacian}. 
\end{defi}

\subsection{Non-autonomous $p$-adic Z\'u\~niga operators}\label{nonAutoEq}
For a time-changing graph $\mathcal{G}_t$ we define its $p$-adic Laplacian as the function $t\mapsto \mathbb{L}(t)$, where
\[
\mathbb{L}(t)=\mathcal{A}(t)-\Gamma(t)\]
where
\[
\mathcal{A}(t)u(x)=\int_{K_N}A(x,y,t)u(y)dy,
\]
and 
\[
\Gamma(t)u(x)=u(x)\gamma(t)(x),
\]
where the kernel function $A(x,y,t)$ and the degree function $\gamma(t)(x)$ are given by 
\[
A(x,y,t)=p^{N}\sum_{I,J\in V(\mathcal{G})} A_{IJ}(t)\Omega_{I,N}(x)\Omega_{J,N}(y),
\]
and
\[
\gamma(t)(x)=\sum_{I\in V(\mathcal{G})}\gamma_I(t)\Omega_{I,N}(x)
\]
with 
$\gamma_{I}(t)=\sum_{J\in V(\mathcal{G})}A_{IJ}(t).$

\begin{defi}
The operators $\mathcal{A}(t)$ and $\mathbb{L}(t)$ are called \emph{non-autonomous $p$-adic Z\'u\~niga operators}. Operator $\mathbb{L}(t)$ is also called \emph{non-autonomous $p$-adic Z\'u\~niga Laplacian}. They are associated with the time-changing graph $\mathcal{G}_t$.
\end{defi}

\begin{teo}\label{stronglyContinuous}
If $t\mapsto A_{I,J}(t)$ is uniformly continuous for each $I,J\in V(\mathcal{G})$ then the $p$-adic Laplacian $t\mapsto \mathbb{L}(t)$ of $\mathcal{G}_t$  is strongly continuous, that is $t\mapsto \mathbb{L}(t)u$ is a continuous function from $[0,\infty)$ to $L^2(K_N,\mathbb{C})$ for every $u\in L^2(K_N,\mathbb{C})$.
\end{teo}
\textit{Proof.} Let $u\in L^2(K_N,\mathbb{C})$  then
\begin{equation*}
    ||\mathcal{A}(t)u-\mathcal{A}(s)u||_{L^2}^2=\int_{K_N}\left| \int_{K_N}(A(x,y,t)-A(x,y,s))u(y)dy\right|^2dx.
\end{equation*}
For $\varepsilon>0$, there is a positive real number $\delta>0$ such that 
\[
|t-s|<\delta \implies|A_{IJ}(t)-A_{IJ}(s)|<\frac{\varepsilon}{n^2}.
\]
Hence
\begin{equation*}
    \begin{split}
        |A(x,y,t)-A(x,y,s)|&\leq\sum_{IJ}|A_{IJ}(t)-A_{IJ}(t)|\\
        &\leq \varepsilon.
    \end{split}
\end{equation*}
Therefore
\begin{equation*}
    \begin{split}
        ||\mathcal{A}(t)u-\mathcal{A}(s)u||_{L^2}^2&\leq\int_{K_N} \left(\int_{K_N}|A(x,y,t)-A(x,y,s)||u(y)|dy\right)^2dx\\
        &\leq \varepsilon^2 \int_{K_N}\left(\int_{K_N}|u(y)|dy\right)^2dx\\
        &\leq \varepsilon^2\Vol(K_N)||u||_{L^2}^2,
    \end{split}
\end{equation*}
where the last inequality is due to Hölder's inequality. 
On the other hand, 
\begin{align*}
|\gamma(t)(x)-\gamma(s)(x)|&=\left|\sum_{I\in
V(\mathcal{I})}(\gamma_I(t)-\gamma_I(s))\Omega_{I,N}(x)\right|
\\
&\leq\sum_{I}\left|\sum_{J}A_{IJ}(t)-A_{IJ}(s)\right|
\\
&\leq \varepsilon
\end{align*}
Hence, 
\begin{equation*}
    \begin{split}
        ||\Gamma(t)u-\Gamma(s) u||_{L^2}^2&\leq \int_{K_N}|\gamma(t)(x)-\gamma(t)(x)|^2|u(y)|^2dx\\
        &\leq \varepsilon^2 ||u||_{L^2}^2
    \end{split}
\end{equation*}
Finally, we have 
\begin{align*}
||\mathbb{L}(t)u-\mathbb{L}(s)u||_{L^2}
&\leq ||\mathcal{A}(t)u-\mathcal{A}(s)u||_{L^2}+||\Gamma(t)u-\Gamma(s) u||_{L^2} 
\\
&\leq \varepsilon||u||_{L^2}\left(\Vol(K_N)^\frac12+1\right)  
\end{align*}
 This ends the proof. \qed \newline
 
In particular the proof of Theorem \ref{stronglyContinuous} shows that $t\mapsto \mathbb{L}(t)$ is continuous in the uniform operator topology. 
\newline

For each $t>0$, $\mathbb{L}(t)$ is a generator of a contraction semigroup $\{e^{s\mathbb{L}(t)}\}$, associated to the Cauchy problem of the heat equation: 
\begin{equation}\label{CPHE}
    \begin{cases}
      \frac{\partial u}{\partial s} (s,x)=\mathcal{L}(t)u(s,x)\\
      u(0,x)=u_0(x)\in L^2(K_n,\mathbb{C})
    \end{cases}\,.
\end{equation}
with $t\geq 0$. 
Now we present the initial value problem called the \textit{evolution system} or the \textit{non autonomous abstract Cauchy problem}
\begin{equation}\label{nAACP}
    \begin{cases}
      \frac{\partial u}{\partial t} (t,x)=\mathbb{L}(t)u(t,x)\\
      u(s,x)=u_0(x)\in L^2(K_n,\mathbb{C})
    \end{cases}\,.
\end{equation}
In the next section we present some results of the theory of semigroups and non autonomous differential equations. 

\subsection{Dirichlet and von Neumann eigenvalues}

Assume that the kernel function $A(x,y,t)$ is symmetric in $(x,y)$, and here also that it is an analytic function for $t\ge0$. Define for fixed $t\ge0$ the $p$-adic edge set of $\mathcal{G}(t)$ as:
\[
\Omega_E(t)=\mathset{(x,y)\in K_N^2\mid A(x,y,t)\neq 0},
\]
as well as the $p$-adic vertex boundary:
\[
\delta_NS(t)=
\mathset{x\in K_N\setminus S\mid\exists y\in S\colon A(x,y,t)\neq0}
\]
as well as the $p$-adic edge boundary of $\mathcal{G}(t)$:
\[
\partial_NS(t)=\mathset{(x,y)\in\Omega_E(t)\mid x\in S\wedge y\notin S}
\]
which could also be called the vertex and edge boundaries of the operator $\mathcal{A}(t)$ defined in Section \ref{nonAutoEq}. Further, define the compact open set
\[
\Omega_S^*(t)=\left(\Omega_E(t)\cap S\times S\right)\cup
\partial_NS(t)\subset\mathbb{Q}_p^2
\]
Now, let $S\subset K_N$ be a compact open subset. Then define:
\[
V(S):=\mathset{I\in V\mid B_{p^{-N}}(I)\cap S\neq\emptyset}.
\]
The space of functions satisfying the \emph{Dirichlet boundary condition} for $S$ is
\[
D_S^*(t)=
\mathset{f\in L^2(K_N)\mid \text{$f(x)=0$ for all $x\in\delta_NS(t)$}},
\]
and the space of functions satisfying the \emph{von Neumann boundary condition} for $S$ is
\begin{align*}
N_S^*(t)=&\\
&\hspace*{-10mm}\mathset{
f\in L^2(K_N)\mid\forall x\in\delta_NS(t)\colon
\int_{\mathset{(x,y)\in \partial_NS(t)}}A(x,y,t)(f(x)-f(y))\,dy=0
},
\end{align*}
again with fixed $t\ge0$.
\newline

In view of \cite{divgrad_p}, the von Neumann boundary condition can be formulated as
\[
D_{\partial S(t),+}f(x)=0
\]
where 
\[
D_{\partial S(t),+}u(x)=
\int_{\mathset{(x,y)\in\partial_NS(t)}}
A(x,y,t)(f(x)-f(y))\,dy
\]
is an advection operator defined similarly as $D_+^\alpha$ from \cite[\S 3]{divgrad_p}, and imitating the normal derivative w.r.t.\ the boundary of $S$ in the classical case of Riemann manifolds.
\newline

The \emph{first Dirichlet eigenvalue} of $\mathbb{L}(t)$ (with respect to $S$) is defined as:
\[
\lambda_{1,\mathbb{L}}^{(D)}(t)
=\inf\limits_{f\neq 0\atop f\in D_S^*(t)}
\frac{p^N\int_{\Omega_S^*(t)}A(x,y,t)\absolute{f(x)-f(y)}^2\,dx\,dy}
{\int_S\absolute{f(x)}^2\gamma(x)(t)\,dx}
\]
The \emph{von Neumann eigenvalue} of $\mathbb{L}(t)$ with respect to $S$ is defined as:
\[
\lambda_{S,\mathbb{L}}^{(N)}(t)
=\inf\limits_{f\neq 0\atop\int_Sf(x)\gamma(x)\,dx=0}
\frac{p^N\int_{\Omega_S^*(t)}A(x,y,t)\absolute{f(x)-f(y)}^2\,dx\,dy}
{\int_S\absolute{f(x)}^2\gamma(x)(t)\,dx}
\]
where it is assumed that $\delta_NS(t)\neq\emptyset$ for this instance of $t\ge0$.
\newline

The Dirichlet and von Neumann eigenvalues of $\mathcal{G}(t)$ w.r.t.\ $S$ will be denoted as
\[
\lambda_{1,\mathcal{G}}^{(D)}(t)\quad\text{and}\quad
\lambda_{S,\mathcal{G}}^{(N)}(t),
\]
respectively.

\begin{lem}\label{vnbc}
Let $\psi\colon S\to\mathbb{C}$ be a Kozyrev wavelet. Then $\psi\in N_S^*(t)$, i.e.\ $\psi$ satisfies the von Neumann boundary condition.
\end{lem}

\begin{proof}
Assume that the Kozyrev wavelet $\psi$ is supported in a ball $B\subset B_{p^{-N}(I)}$ for some vertex $I\in V(S)$. Then
\begin{align*}
\int_{(x,y)\in\partial_NS(t)}&A(x,y,t)(\psi(x)-\psi(y))\,dy
\\
&=\int_{(x,y)\in\partial_NS(t)}A(x,y,t)\,dy\,\psi(x)
\\
&-\int_{(x,y)\in\partial_NS(t)}A(x,y,t)\psi(y)\,dy
\\
&=0-0=0
\end{align*}
where the left term vanishes, because $x\in\delta_NS(t)$, and
the right integral vanishes, because $\psi$ is a Kozyrev wavelet, and   from Kozyrev's theorem that
\[
\int_{K_N}\psi(x)\,dx=0,
\]
cf.\ e.g.\ \cite[Thm.\ 3.29]{XKZ2018} or \cite[Thm.\ 9.4.2]{AXS2010}.
\end{proof}

\begin{teo}\label{boundaryEV}
Let $S\subseteq K_N$ be compact open, and assume that
$\delta S(t)$ is non-empty for $t\in [0,a)$ with $a>0$. Then the following statements hold true:
\begin{enumerate}
\item $\lambda_{1,\mathbb{L}}^{(D)}(t)\le\min\mathset{\lambda_{1,\mathcal{G}}^{(D)}(t),\hat{\gamma}_S(t)}\le1,$
\item
$\lambda_{S,{\mathbb{L}}}^{(N)}(t)\le
\min\mathset{\lambda_{S,\mathcal{G}}^{(N)}(t),\hat{\gamma}_S(t)
}$,
\end{enumerate}
where
\[
\hat{\gamma}_S(t)=\min\limits_{I\in V(S)}\frac{\deg_{\partial V(S)(t)}(I)}{\int_S\gamma(t)(x)\,dx}
\]
and $\deg_{\partial V(S)(t)}(I)$
is the degree of vertex $I$ in the edge boundary $\partial V(S)(t)$, viewed as a subgraph of $\mathcal{G}(t)$,  for $t\in[0,a)$.
\end{teo}

\begin{proof}
The inequalities $\lambda_{1,\mathbb{L}}^{(D)}(t)\le\lambda_{1,\mathcal{G}}^{(D)}(t)$ and $\lambda_{S,\mathbb{L}}^{(N)}(t)\le\lambda_{S,\mathcal{G}}^{(N)}(t)$ are obvious. Now, let $\psi$ be a Kozyrev wavelet supported in $B_{p^{-r}}(I)\subset B_{p^{-N}}(I)\cap S$ for some vertex $I\in V(S)$. 
It clearly satisfies the Dirichlet boundary condition,  and, by Lemma \ref{vnbc}, also the von Neumann boundary condition.
Hence, the Dirichlet and the von Neumann eigenvalues are both less than or equal to:
\[
\frac{\sum\limits_{J\in\delta V(S)(t)}A_{IJ}}
{\int_S\gamma(x)(t)\,dx}
=\frac{\deg_{\partial V(S)(t)}(I)}{\int_S\gamma(x)(t)\,dx}
\]
where $\delta V(S)(t)$ is the vertex boundary of the set $V(S)\subset V$ in $\mathcal{G}(t)$. This now shows the first inequality in \emph{1.}, and that of \emph{2.}
The remaining inequality in \emph{1}.\ is given in \cite[\S 8.4]{Chung1997}. This proves the theorem.
\end{proof}

The following examples show that the inequalities can be strict:

\begin{exm}
Let $\mathcal{G}(t)$ be the complete graph on three vertices $a,b,c$, and let $S=B_{p^{-N}}(a)\sqcup B_{p^{-N}}(b)$.
Then, according to \cite[\S 8.4]{Chung1997}, $\lambda_{1,\mathcal{G}}^{(D)}(t)$ is an eigenvalue of the submatrix of the graph-Laplacian of  $\mathcal{G}(t)$ spanned by the vertices $a,b$.
This is the matrix
\[
L_S=\begin{pmatrix}
1&-\frac12\\
-\frac12&1
\end{pmatrix}
\]
with eigenvalues $\frac32$ and $\frac12$.
Hence, $\lambda_{1,\mathcal{G}}^{(D)}(t)=\frac12$. But since the degree of each vertex of $\mathcal{G}(t)$ equals $2$, it follows that
\[
\lambda_{1,\mathbb{L}}^{(D)}(t)\le\hat{\gamma}_S(t)=\frac14<\frac12=\lambda_{1,\mathcal{G}}^{(D)}(t)
\]
for given $t\ge0$.
\end{exm}

\begin{exm}
Let $\mathcal{G}(t)$ be the complete graph on two vertices $a,b$, and let $S=B_{p^{-N}}(a)$. Then any non-trivial function $f$ constant on the vertex balls $B_{p^{-N}}(a)$ and $B_{p^{-N}}(b)$ satisfying the von Neumann boundary condition, must satisfy
$f(a)= 0$ and thus
\[
\int_S\absolute{f(x)}^2\gamma(x)(t)\,dx=0
\]
This means that 
\[
\lambda_{S,\mathcal{G}(t)}^{(N)}(t)=\infty
\]
and consequently
\[
\lambda_{S,\mathbb{L}}^{(N)}(t)<\lambda_{S,\mathcal{G}}^{(N)}(t)
\]
because $\hat{\gamma}_S(t)<\infty$.
\end{exm}

\section{Non Autonomous Cauchy Problem}

Consider the $p$-adic non autonomous Cauchy problem for a finite graph.
Most of the underlying theory for this section is taken from
\cite{EN2006,Fattorini1983}.

\subsection{Review of the abstract Cauchy problem}

Let $A(t)$ be a family of operators with $t\in\mathbb{R}$ on a Banach space $X$. The following problem is addressed:

\begin{equation}\label{nACP_abstract}
    \begin{cases}
      \frac{\partial u}{\partial t} (t)=A(t)u(t)\\
      u(s)=x\in X
    \end{cases}
\end{equation}
with $t,s \in \mathbb{R}$ and $t\geq s$. The system $(\ref{nACP_abstract})$ is called (nACP).

\begin{defi}
A continuous function $u:[s,\infty)\rightarrow X$ is called a (strict) \textbf{solution} of (\ref{nACP_abstract}), if $u\in C^1([s,\infty),X)$, $u(t)\in D(A(t))$ for all $t\geq s$, $u(s)=x$, and $\frac{\partial u}{\partial t}=A(t)u$ for $t\geq 0$.
\end{defi}
The solution of non-autonomous Cauchy problem on some Banach space $X$ can be given by a so called \textit{evolution family} which can be defined as follows, cf.\ \cite{Schnaubelt2006}:
\begin{defi}
A family $(U(t,s))_{t\geq s}$ of linear operators on a Banach space $X$ is called an \textbf{evolution family} for (\ref{nACP_abstract}), if 
\begin{enumerate}
    \item $U(t,r)U(r,s)=U(t,s)$, $U(t,t)=id$, $t\geq s \in \mathbb{R}$.
    \item The mapping $(s,t)\mapsto U(t,s)$ is strongly continuous. \item $||U(t,s)||\leq M e^{\omega(t-s)}$ for some $M\geq 1$, $\omega\in \mathbb{R}$, $t\geq s \in \mathbb{R}$.
    \item For $s\in \mathbb{R}$ the regularity space
    \[
    Y_s:=\{y\in X:[s,\infty) \ni t\rightarrow U(t,s)y \ solves \ (nACP)\}
    \]
\end{enumerate}
is dense in X, where $U(t,s)y$ is the unique solution of (nACP). 
\end{defi}
If there exists a solving evolution family, then (\ref{nACP_abstract}) is called \textit{well-posed}. We have the following Theorem.

\begin{teo}\label{fattoriniTheorem}
Let $X$ a Banach space and for every $t>0$ let $A(t)$ be a bounded linear operator on $X$. If the function $t\mapsto A(t)$ is strongly continuous for $0\le t<T$
(that is, $t\mapsto A(t)u$ is continuous for each $u\in X$), then for every $x\in X$ the problem (\ref{nACP_abstract}) is well-posed.
\end{teo}

\begin{proof}
\cite[Thm.\ 7.1.1]{Fattorini1983}.
\end{proof}

We can apply the above Theorem to the problem (\ref{nAACP}) for time-changing graphs.
\begin{prop}
The system (\ref{nAACP}) is well-posed.
\end{prop}
\textit{Proof.} This is a direct consequence of Theorem
\ref{stronglyContinuous}. \qed \newline

\begin{rem}\label{solutionProperties}
If we let $U(t,s)x=u(t)$, where $u(t)$ is the solution of (\ref{nAACP}), then the following holds  true: 
\begin{enumerate}
    \item $||U(t,s)||\leq e^{\int_{s}^{t} ||\mathbb{L}(\tau)||d\tau}$.
    \item $(t,s)\mapsto U(t,s)$ is continuous in the uniform operator topology for $s\leq t$.
    \item $\partial U(t,s)/\partial t= \mathcal{L}(t) U(t,s)$.
    \item $\partial U(t,s)/\partial s=-U(t,s)\mathcal{L}(s)$. 
\end{enumerate}
This is due to the fact that $t\mapsto \mathbb{L}(t)$ is continuous in the uniform norm topology, cf.\ \cite[Theorem 5.2]{Pazy1983}.
\end{rem}
In general, the solution does not have a closed form, but in special cases it does, as the following example shows.

\begin{exm} Under the hypothesis of Theorem \ref{fattoriniTheorem}, assume that $A(t)$ and $\int_{s}^{t} A(\tau) d\tau$ commute for each $s,t$. Then 
\[
U(t,s)=e^{\int_{s}^{t}A(\tau) d\tau}
\]
with the integral $\int A(\tau) d\tau$ being interpreted "elementwise". 
\end{exm}

\begin{proof}
\cite[Example 7.1.9]{Fattorini1983}.
\end{proof}

\subsection{Reviewing  the autonomous case for graphs}

First, we review how to solve the heat equation in the autonomous case:

\begin{equation}\label{HEAC}
    \begin{cases}
      \frac{\partial u}{\partial t} (x,t)=\mathbb{L}u(x,t)\\
      u(0,x)=u_0(x)\in L^2(K_n,\mathbb{C})
    \end{cases}\,,
\end{equation}
where $\mathbb{L}$ is interpreted as $\mathbb{L}=\mathbb{L}(t_0)$ for some $t_0\geq0$. The following Theorem allows us to apply the Fourier Method (or the separation of variables method) to solve the system (\ref{HEAC}).

\begin{teo}[Z\'u\~niga-Galindo, 2020]\label{somethingKnown}
The elements of the set: 
\[
\Spec(L) \bigsqcup \{-\gamma_I: I\in V(\mathcal{G})\}
\]
 are the eigenvalues of of $\mathbb{L}$. The corresponding eigenfunctions are given by the following infinite set 
\[
\left\{ \frac{\varphi_\mu}{||\varphi_\mu||_2}: \mu\in \Spec(L)\right\}
\bigsqcup 
\left\{
\Psi_{j,I,r}:I\in V(\mathcal{G}), j\in \{1,...,p-1\}, r\in \mathbb{Z}, r\geq N\right\},
\]
where the functions $\varphi_\mu(x)$ are defined by: 
\begin{equation}
  \varphi_\mu(x)=\sum_{J\in V(\mathcal{G})} \varphi_\mu(J)\Omega(p^{-N}|x-J|_p)
\end{equation}
and the vector $(\varphi_\mu(J))_{J\in V(\mathcal{G})}$ is an eigenvector of the (negative semi-definite) Laplacian matrix $L$ associated with adjacency matrix $A$, corresponding to $\mu\in\Spec(L)$, and $\Psi_{r,I,j}$ are the Kozyrev functions of the form (\ref{KozyrevWavelet}). Furthermore we have the following orthogonal decomposition 
\begin{equation}\label{orthoDecomp}
    L^2(K_N,\mathbb{C})=X_N
    \oplus \mathcal{L}_0^2(K_N).
\end{equation}

\end{teo}

\begin{proof}
\cite[Thm.\ 10.1]{ZunigaNetworks}.
\end{proof}

According to \cite[Thm.\ 4.2]{ZunigaNetworks}, the Cauchy problem of the heat equation (\ref{HEAC}), with $t\geq 0$ has a solution which is unique. The solution is given by the semigroup $\{e^{t\mathbb{L}}\}$ generated by $\mathbb{L}$. By the decomposition (\ref{orthoDecomp}), we propose a solution of the form 
\begin{equation}\label{expansion}
    u(x,t)=\sum_{\mu\in\Spec(L)} c_{\mu}(t)\varphi_{\mu}(x)+\sum_{\supp(\Psi_{j,I,r}) \subset K_N} c_{r,I,j}(t)\psi_{j,I,r}(x),
\end{equation}
where the infinite set $\{\Psi_{j,I,r}|\supp(\Psi_{j,I,r}) \subset K_N\}$ is described in Theorem \ref{somethingKnown}. Hence, the solution is an infinite sum, and the convergence takes place in $L^2(K_N)$. After inserting this into (\ref{HEAC}) we obtain 
\begin{equation*}
    \begin{cases}
      \frac{\partial c_\mu}{\partial t} =\mu c_{\mu}\\
      \frac{\partial c_{r,I,j}}{\partial t}=-\gamma_{I}c_{j,I,r}
    \end{cases}\,.
\end{equation*}
where $\mu$ and $\gamma_I$ are the eigenvalues associated to $\varphi_\mu$ and $\psi_{j,I,r}$, respectively.
\newline

This proves that 
\begin{equation}\label{expansion2}
    e^{t\mathbb{L}}u_0(x)=u(x,t)=\sum_{\mu\in \Spec(L)} C_{\mu}e^{\mu t}\varphi_{\mu}(x)+\sum_{\supp(\Psi_{j,I,r}) \subset K_N} C_{j,I,r}e^{-\gamma_I t}\psi_{j,I,r}(x).
\end{equation}
with $C_I$ and $C_{j,I,r}$ is uniquely determined by the initial conditions 
\begin{equation}\label{initialCond}
    u_0(x)=u(x,0)=\sum_{\mu\in\Spec(L)} C_{\mu}\varphi_{\mu}(x)+\sum_{\supp(\Psi_{j,I,r}) \subset K_N} C_{j,I,r}\psi_{j,I,r}(x),
\end{equation}
and solution (\ref{expansion2}) is an infinite convergent sum.

\subsection{Approximating the non-autonomous case for graphs}

We now want to use the Trotter product formula to approximate the operator $U(t,s)$, i.e., the solutions of (\ref{nAACP}). As approximating operators we choose the product 
\[
\prod_{k=1}^{n} e^{t/n\mathcal{L}(s+kt/n)}=e^{t/n\mathcal{L}(s+t)}\cdot e^{t/n\mathcal{L}(s+t-t/n)}\dots e^{t/n\mathcal{L}(s+t/n)}.
\]
We have the following proposition:

\begin{prop}\label{uniformConvergence}
Let $\{A(t)\}_{t\in \mathbb{R}}$ be a family of bounded and strongly continuous operators on a Banach space $X$, and let $U(t,s)$ be its respective evolution family. Then 
\begin{equation}
    \lim_{n\rightarrow \infty} \prod_{k=1}^{n}e^{t/nA(s+kt/n)}x=U(s+t,s)x
\end{equation}
for all $x\in X$ and uniformly for $s$ and $t$ in compact intervals of $\mathbb{R}$ and $\mathbb{R}_{+}$, respectively. 
\end{prop}

\begin{proof}
\cite[Ch.\ III.5.9, Prop.]{EN2006}.
\end{proof}

Now we outline how to apply Proposition \ref{uniformConvergence} to approximate the solution of (\ref{nAACP}). For $0<b$, and $\varepsilon>0$, we have 
\[
||e^{t/n\mathbb{L}(t)}\cdot e^{t/n\mathbb{L}(t-t/n)}\dots e^{t/n\mathbb{L}(t/n)}u_0(x)-U(t,0)u_0(x)||_{L^2}<\varepsilon
\]
with $0 \leq t<b$, and some $n\in \mathbb{N}$. Take $0\leq t_0 <b$ fixed, hence, in order to compute the approximate (respect to the $L^2$ norm) solution we have to solve (\ref{HEAC}) by iterating the initial conditions, or in other words we have to compute $e^{t_0/n\mathbb{L}(t_0)}\cdot e^{t_0/n\mathbb{L}(t_0-t_0/n)}\dots e^{t_0/n\mathbb{L}(t_0/n)}u_0(x)$. We can view this approximation like snapshots of the system at $n\ge0$ discrete time points. 
We call this the \emph{temporal discretisation}, and we fix some notation:

\begin{enumerate}
    \item Let $\mu(t_0)$ be an eigenvalue of the matrix $L(t_0)$ and $\varphi_\mu(x,t_0)$ the associated eigenfunction.
    \item Let $\{e^{t\mathbb{L}(t_0)}\}_{t\geq 0}$ be the semigroup generated by the operator $\mathbb{L}(t_0)$ for some $t_0\geq 0$. 
    \item Let $\gamma_{I}(t_0)=\sum_{J}A_{IJ}(t_0)$
\end{enumerate}
Therefore, for a fixed $t_0$, the solution of the Cauchy problem in the autonomous case (\ref{HEAC}) associated to the operator $\mathbb{L}(t_0)$, with initial condition $u(x,0)=u_0(x)$ as in (\ref{initialCond}), is given by 
\begin{align*}
e^{t\mathbb{L}(t_0)}&u_0(x)=u(x,t)
\\
&=\sum_{\mu=\mu(t_0)\in \Spec(L(t_0))} C_{\mu}(t_0)e^{\mu(t_0) t}\varphi_{\mu}(x,t_0)
\\
&+\sum_{r,I,j} C_{j,I,r}(t_0)e^{-\gamma_I(t_0) t}\psi_{j,I,r}(x).
\end{align*}
Now it is clear that 
\begin{align}\nonumber
     e^{t_0/n\mathbb{L}(t_0/n)}u_0(x)
     &=\sum_{\mu=\mu(t_0)\in\Spec(L(t_0))} C_{\mu}(t_0/n)e^{\mu(t_0/n) t_0/n}\varphi_{\mu}(x,t_0/n)
     \\\label{expansion3}
&     +\sum_{j,I,r} C_{r,I,j}(t_0/n)e^{-\gamma_I(t_0/n) t_0/n}\psi_{j,I,r}(x).
\end{align}
The next step is to compute $e^{t_0/n\mathbb{L}(2t_0/n)}e^{t_0/n\mathbb{L}(t_0/n)}u_0(x)$. For this we need to find the solution to the Cauchy problem (\ref{HEAC}) associated to the operator $\mathbb{L}(2t_0/n)$ with initial condition $u_1(x)=e^{t_0/n\mathbb{L}(t_0/n)}u_0(x)$, and evaluate the solution $t\mapsto e^{t\mathbb{L}(2t_0/n)}e^{t_0/n\mathbb{L}(t_0/n)}u_0(x)$ in $t=t_0/n$.
\newline

Define for a given adjacency matrix $A(t)$ the corresponding weighted graph $G_{A(t)}$ whose adjacency matrix is $A(t)$. The corresponding $L^2$-space will  be denoted as $L^2(G_{A(t)},\mathbb{C})$. It is isomorphic to $\mathbb{C}^{\absolute{V(\mathcal{G})}}$.
\newline

To find the solution of the form (\ref{expansion}) we need to express the initial condition $u_1(x)=e^{t_0/n\mathbb{L}(t_0/n)}u_0(x)$ as an element of $L^2(G_{A(2t_0/n)},\mathbb{C})\oplus \mathcal{L}_0^2(K_N)$, 
that is 
\[
e^{t_0/n\mathbb{L}(t_0/n)}u_0(x)=\sum_{\mu\in \Spec(L(t_0))} B_\mu\varphi_\mu(x,2t_0/n)+\sum_{j,I,r} B_{j,I,r}\psi_{r,I,j}(x).
\]
Then, after computing the coefficients $B_\mu$ and $B_{j,I,r}$ the desired solution will be given by 
\begin{equation}
    \begin{split}
        e^{t\mathbb{L}(2t_0/n)}&e^{t_0/n\mathbb{L}(t_0/n)}u_0(x)= u_1(x,t)
        \\\label{expansion4}
        &\hspace*{-.5cm}=\sum_{\mu\in\Spec(L(t_0))} B_{\mu}e^{\mu(2t_0/n) t}\varphi_{\mu}(x,2t_0/n)+\sum_{j,I,r} B_{j,I,r}e^{-\gamma_I(2t_0/n) t}\psi_{j,I,r}(x).
    \end{split}
\end{equation}
Comparing (\ref{expansion3}) and (\ref{expansion4}), it is clear that $B_{j,I,r}=C_{j,I,r}e^{-\gamma_I(t_0/n)t_0/n}$. This is due to the fact that in decomposition (\ref{orthoDecomp}) the space $\mathcal{L}_0^{2}(K_N)$ remains constant in time. Now we need to compute the constants $B_\mu$. For this, note that 
\[
\omega:=\sum_{\mu\in\Spec(L(t_0))} B_\mu\varphi_\mu(x,2t_0/n)=\sum_{\mu\in\Spec(L(t_0))} C_{\mu}e^{\mu(t_0/n) t_0/n}\varphi_{\mu}(x,t_0/n)
\]
Note that  for each $s\geq 0$, the set 
\[
\mathset{\varphi_\mu(x,s)=\varphi_{\mu(s)}(x)\mid\mu(s)\in\Spec(L(s))}
\]
is a basis for the finite dimensional space $X_N$. Hence, the coefficients $B_\mu$ are given by the basis change matrix from the basis $\{\varphi_{\mu}(x,t_0/n)\}$ to the basis $\{\varphi_\mu(x,2t_0/n)\}$. The basis change matrix $M=[M_{\lambda,\mu}]$ is given by the following equalities: 
\[
\varphi_{\mu}(x,2t_0/n)=\sum_{\lambda\in V(\mathcal{G})}M_{\lambda,\mu}\varphi_{\lambda}(x,t_0/n).
\]
In general, the basis change matrix from the basis $\{\varphi_\mu(x,s)\}$ to the basis $\{\varphi_\mu(x,t)\}$ $M$ will be given by 
\[
M=M(s)^{-1}M(t),
\]
where the matrices $M(\tau)$ are the \textit{modal} matrices of the Laplacian matrix, that is, the $\mu$-column of $M(\tau)$ is given by the eigenvector $[\varphi_\mu(K,\tau)]_{K\in V(\mathcal{G})}$. Then, it is clear that the coefficients $B_\mu$ are given by
\[
[B_\mu]_{\mu}=M^{-1}[C_{\mu}e^{\mu(t_0/n) t_0/n}]_{\mu}.
\]
Thus,
\[
[B_\mu]_{\mu}=M(2t_0/n)^{-1}M(t_0/n)\diag(e^{\mu(t_0/n) t_0/n})[C_{\mu}]_{\mu}
\]
and 
\[
B_{j,I,r}=C_{j,I,r}e^{-\gamma_I(t_0/n)t_0/n}.
\]
The process continues by solving the Cauchy problem for the operator $\mathbb{L}(3t_0/n)$ with initial condition $e^{t_0/n\mathbb{L}(2t_0/n)}e^{t/n\mathbb{L}(t_0/n)}u_0(x)$. In order to do so, we repeat the process described above. That is we take as initial condition 
\begin{align*}
u_1(x,t_0/n)&=\sum_{\mu\in\Spec(L(t_0))} B_{\mu}e^{\mu(2t_0/n) t_0/n}\varphi_{\mu}(x,2t_0/n)
\\
&+\sum_{j,I,r} C_{j,I,r}e^{-\gamma_I(t_0/n)t_0/n-\gamma_I(2t_0/n) t}\psi_{j,I,r}(x)
\end{align*}
and express it as an element of $L^2(G_{A(3t_0/n)},\mathbb{C})\oplus \mathcal{L}_0^2(K_N)$.
Now,
\[
u_1(x,t_0/n)=\sum_{\mu\in\Spec(L(t_0))} D_\mu\varphi_\mu(x,3t_0/n)+\sum_{j,I,r} D_{j,I,r}\psi_{r,I,j}(x).
\]
Again, we obtain the coefficients $D_\mu$ and $D_{j,I,r}$ by:
\[
[D_\mu]_\mu=M(3t_0/n)^{-1}M(2t_0/n)\diag(e^{\mu(2t_0/n)t_0/n})[B_\mu]_\mu
\]
and 
\[
D_{j,I,r}=B_{j,I,r}e^{-\gamma_I(2t_0/n)t_0/n}.
\]
We can simplify these expressions by putting the value of $B_\mu$ and $B_{j,I,r}$ in the above equalities, and we get 
\[
[D_\mu]_\mu=M(3t_0/n)^{-1}M(t_0/n)\diag(e^{\mu(t_0/n)+\mu_I(2t_0/n)})[C_\mu]_\mu
\]
and 
\[
D_{j,I,r}=C_{j,I,r}e^{-\gamma_I(t_0/n)t_0/n-\gamma_I(2t_0/n) t}.
\]
In the end, we obtain the function $e^{t_0/n\mathbb{L}(t_0)}\dots e^{t_0/n\mathbb{L}(t_0/n)}u_0(x)\in L^2(K_N)$, of the form 
\begin{equation*}
    \begin{split}
        e^{t_0/n\mathbb{L}(t_0)}&\cdot e^{t_0/n\mathbb{L}(t_0-t_0/n)}\dots e^{t_0/n\mathbb{L}(t_0/n)}u_0(x)=\sum_{\mu\in\Spec(L(t_0))}Q_\mu(n,t_0)\varphi_\mu(x,t_0)\\&+\sum_{\supp(\Psi_{j,I,r})\subset K_N}C_{r,I,j}(t_0/n)e^{-\frac{t_0}{n}\sum_{k=1}^n\gamma_I(k\frac{t_0}{n})}\Psi_{j,I,r}(x),
    \end{split}
\end{equation*}
where the coefficients $Q_\mu(n,t_0)$ are given by the equation
\[
[Q_\mu(n,t_0)]_\mu=M(t_0)^{-1}M(t_0/n)\diag(e^{\frac{t_0}{n}\sum_{k=1}^{n}\mu(k\frac{t_0}{n})})[C_\mu(t_0/n)]_\mu
\]
For simplicity, we rename the approximation above: 
\begin{align*}
U_n(t,0)u_0(x)&=\sum_{\mu\in\Spec(L(t_0))}Q_\mu(n,t)\varphi_\mu(x,t)
\\
&+\sum_{\supp(\Psi_{j,I,r})\subset K_N}C_{r,I,j}(t/n)e^{-\frac{t}{n}\sum_{k=1}^n\gamma_I(k\frac{t}{n})}\Psi_{j,I,r}(x)
\end{align*}
Then, for a general initial condition $s$, we obtain: 
\begin{align*}
U_n(t,s)u_0(x)&=\sum_{\mu\in\Spec(L(t_0))}Q_\mu(n,t,s)\varphi_\mu(x,t_0)
\\
&+\sum_{\supp(\Psi_{j,I,r})\subset K_N}C_{r,I,j}(s+\frac{t-s}{n})e^{-\frac{t-s}{n}\sum_{k=1}^n\gamma_I(s+k\frac{t-s}{n})}\Psi_{j,I,r}(x),
\end{align*}
where the coefficients $Q_\mu(n,t,s)$ are computed by the following equation: 
\[
[Q_\mu(n,t,s)]=M(t)^{-1}M\left(s+\frac{t-s}{n}\right) \diag\left(e^{\frac{t-s}{n}\sum_{k=1}^{n}\mu(s+k\frac{t-s}{n})}\right)[C_\mu(s+\frac{t-s}{n})]_\mu
\]
where the numbers $C_\mu(s+\frac{t-s}{n})$ are the Fourier coefficients of the initial condition at time $s+\frac{t-s}{n}$. Note that in $L^2(K_N)$ we have the following convergence:
\[
\lim_{n\rightarrow\infty}U_n(t,s)u_0(x)=U(t,s)u_0(x),
\]
and we can expand the function $U(t,s)u_0(x)\in L^2(K_N)$ by its Fourier series using the decomposition $L^2(G_{A(t_0)},\mathbb{C})\oplus \mathcal{L}_0^2(K_N)$:
\[
U(t,s)u_0(x)=\sum_{\mu\in\Spec(L(t_0))}Q_\mu(t,s)\varphi_\mu(x,t_0)+\sum_{\supp(\Psi_{j,I,r})\subset K_N}Q_{r,I,j}\Psi_{j,I,r}(x).
\]
Hence, by Parseval's equality we obtain:
\[
\lim_{n\rightarrow\infty} Q_\mu(n,t,s)=Q_\mu(t,s)
\]
and 
\[
\lim_{n\rightarrow\infty}C_{r,I,j}\left(s+\frac{t-s}{n}\right)e^{-\frac{t-s}{n}\sum_{k=1}^n\gamma_I(s+k\frac{t-s}{n})}=C_{r,I,j}(s) e^{-\int_{s}^{t}\gamma_I(\tau)d\tau},
\]
where the last equality occurs by the continuity of $t\mapsto \gamma_I(t)$. On the other hand, the first limit  is more complicated, because it depends on the continuity of the matrix function $t\mapsto M(t)$, more specifically, on the continuity of the eigenvectors: $t\mapsto \varphi_\mu(t)$, where the vector $\varphi_\mu=[\varphi_\mu(K,t)]_K$ is the eigenvector of the Laplacian matrix $L(t)$ associated to $\mu(t)\in\Spec(L(t))$. \newline

Let $t\mapsto B(t)$ be a  matrix function. If $B(t)$ is real symmetric, then we can find an orthonormal basis $\{\varphi_\mu(t)\}_\mu$  of $\mathbb{C}^N$ formed by eigenvectors of $B(t)$. It can be shown that if the matrix $B(t)$ depends analytically on $t$ inside a real interval, then we can chose an orthonormal basis $\{\varphi_\mu(t)\}_\mu$ that depends analytically on $t$
\cite[Ch.\ II.4]{Kato1980}, 
and in particular, continuously on $t$. If  the analyticity assumption is deleted, the eigenvectors may not depend continuously on $t$. Therefore, if we suppose that the adjacency matrix $A(t)$ depends analytically on $t$, then we have the following:

\begin{equation*}
    \begin{split}        \lim_{n\rightarrow \infty} [Q_\mu(n,t,s)]_\mu&=\lim_{n\rightarrow \infty}M(t)^{-1}M\left(s+\frac{t-s}{n}\right) 
        \\
        &\cdot \diag\left(e^{\frac{t-s}{n}\sum_{k=1}^{n}\mu(s+k\frac{t-s}{n})}\right)\left[C_I\left(s+\frac{t-s}{n}\right)\right]_I\\
        &=M(t)^{-1}M(s)\diag\left(e^{\int_{s}^{t} \mu(\tau)d\tau}\right)[C_\mu(s)]_\mu=[Q_\mu(t,s)]_\mu.
    \end{split}
\end{equation*}
But, in order to have the above limit, it is sufficient to suppose that the eigenvectors $\varphi_\mu(t)$ are right continuous at the initial condition $s\ge t_0$. 
\newline

This proves the following:

\begin{teo}\label{CauchyProblem}
If the eigenvectors of the Laplacian matrices $L(t)$ are right continuous at some $s\geq0$, then the evolution family of the Cauchy problem (\ref{nAACP}) (with initial condition at $s$) is given by 
\begin{align*}
U(t,s)u_0(x)&=\sum_{\mu\in \Spec(L(t_0))}Q_\mu(t,s)\varphi_\mu(x,t)
\\
&+\sum_{\supp(\Psi_{j,I,r})\subset K_N} C_{r,I,j}(s) e^{-\int_{s}^{t}\gamma_I(\tau)d\tau} \Psi_{j,I,r}(x),
\end{align*}
where $t\ge s$,
\[
[Q_\mu(t,s)]=M(t)^{-1}M(s)\diag\left(e^{\int_{s}^{t} \mu(\tau)d\tau}\right)[C_\mu(s)]_\mu\,,
\]
and the coefficients $C_\mu(s)$ and $C_{r,I,j}(s)$ are uniquely determined by the initial condition
\[
u_0(x)=\sum_{\mu\in\Spec(L(t_0))}C_\mu(s)\varphi_\mu(x,s)+\sum_{\supp(\Psi_{j,I,r})\subset K_N} C_{r,I,j}(s)  \Psi_{j,I,r}(x).
\]
In particular, this result holds for all $s\geq 0$ if the adjacency matrix $A(t)$ depends analytically on $t$.
\end{teo}

\begin{rem}
Note that this is consistent with the autonomous problem, that is, when $L(t)=L$ is constant. In this case, the matrices $M(t)$ and the functions $t\mapsto\mu(t)$, $t\mapsto \varphi_\mu(t)$ are constant, thus the solution is given by 
\begin{align*}
U(t,s)u_0(x)&=\sum_{\mu\in\Spec(L)}C_\mu(s)e^{(t-s)\mu}
\\
&+\sum_{\supp(\Psi_{j,I,r})\subset K_N} C_{r,I,j}(s) e^{-(t-s)\gamma_I} \Psi_{j,I,r}(x) 
\end{align*}
for $t\ge s\ge0$.
\end{rem}

\begin{cor}\label{error}
Assuming analyticity of $L(t)$, the first-order approximation error in the temporal discretisation is $O\left(\frac{1}{n}e^{\frac{1}{n}}\right)$.
\end{cor}

\begin{proof}
For the temporal discretisation, observe that by replacing each analytic function with its first order approximation leads to 
an error of $O(1/n)$ in each entry of the eigenvectors and eigenvalues of $L(t)$, and an error of $O(e^{1/n})$
in each exponential integral, because the temporal discretisation is in this case a Riemann sum. Hence, the
error in each coefficient of the solution is given by the 
product error, and this is in first order $O\left(\frac{1}{n} e^{\frac{1}{n}}\right)$, which proves the assertion. 
\end{proof}

%
%
%
%
%
%

\section{Inhomogeneous Markov processes and Feller evolution.}

\subsection{Reviewing inhomogeneous Markov processes}
In this section we review some definitions and results necessary to prove that the non-autonomous abstract Cauchy problem (\ref{nAACP}) has attached to it an inhomogeneous Markov process. 
\newline

In the sequel $E$ denotes a separable complete metrizable topological Hausdorff space. In other words $E$ is a Polish space. The space $C_b(E)$ is the space of all complex valued bounded continuous functions. The space $C_b(E)$ is equipped with the uniform norm: $||\cdot||_{\infty}$, and the strict topology $\mathcal{T}_\beta$.
It is considered as a subspace of the bounded Borel measurable functions $L^{\infty}(E)$, also endowed with the supremum norm. \newline

Let $(\Omega,\mathcal{F})$ be a measurable space. A filtration $\mathcal{F}_t$, $t\in [0,T]$, is a family of sub-$\sigma$-algebras of $\mathcal{F}$ such that if $0\leq t_1 \leq t_2 \leq T$, then $\mathcal{F}_{t_1}\subset \mathcal{F}_{t_2}$. A two-parameter filtration $\mathcal{F}_t^{\tau}$, $0\leq \tau \leq t \leq T$, is a family of sub-$\sigma$-algebras of $\mathcal{F}$ that is increasing in $t$ and decreasing in $\tau$. Let $X_s$, $s\in [0,T]$ be a stochastic process, with state space $(E, \mathcal{E})$. The $\sigma$-algebra $\sigma(X_s:0\leq s\leq t)$ is the smallest $\sigma$-algebra such that all the state variables $X_s$ with $0\leq s\leq t$ are measurable.
The process $X_s$, $s\in [0,T]$, generates the following two-parameter filtration: 
\[
\mathcal{F}_t^{\tau}=\sigma(X_s:\tau\leq s \leq t)
\]
with $0\leq \tau \leq t \leq T$. 
\newline

We recall some definitions:

\begin{defi}
A transition probability funcition $P(r,x;s,A)$, where $0\leq r < s \leq T$, $x\in E$, and $A\in \mathcal{E}$, is a nonnegative function for which the following condition hold: 
\begin{itemize}
    \item For fixed $r,s$ and $A$, $P$ is a nonnegative Borel measurable function on $E$. 
    \item For fixed $r,s$ and $x$, $P$ is a Borel measure on $\mathcal{E}$. 
    \item $P(r,x;s,E)=1$ for all $r,s$ and $x$. 
    \item $P(r,x;s,A)=\int_{E}P(r,x;y,dy)P(u,y;s,A)$ for all $r<u<s$, and $A$. 
\end{itemize}
\end{defi}
The number $P(r,x;s,A)$ can be interpreted as the probability of the following event: The random system located at $x\in E$ at time $r$ hits the target $A\subset E$ at time $s$. 

\begin{defi}
The process
\[
\{
(\Omega,\mathcal{F}_T^{\tau},\mathbb{P}_{\tau,x}), (X(t):T\geq t\geq 0),(E,\mathcal{E})\}
\]
is called a time-inhomogeneous Markov process if 
\[
\mathbb{E}_{\tau,x}\left[f(X(t))|\mathcal{F}_s^{\tau}\right]=\mathbb{E}_{s,X(s)}[f(X(t))]
\]
$\mathbb{P}_{\tau,x}$-almost surely. Here $f$ is a bounded Borel measurable function defined on the state space $E$ and $\tau\leq s \leq t \leq T$. 

\end{defi}
Let $P$ be a transition probability function. It is possible to construct a filtered measurable space $(\Omega, \mathcal{F}, \mathcal{F}_t^{\tau})$, a family of probability measures $\mathbb{P}_{\tau,x}$, $x\in E$, $\tau \in [0,T]$, on the space $(\Omega, \mathcal{F}_T^{\tau})$, and a Markov process $X_t$, $t\in [0,T]$ such that 
\[
\mathcal{F}_t^{\tau}=\sigma(X_s:\tau \leq s \leq t)
\]
with $0\leq \tau \leq t \leq T$ and 
$$\mathbb{P}_{\tau,x}(X_t\in A)=P(\tau,x;t,A)$$
with $0\leq \tau \leq t \leq T, $ $A\in \mathcal{E}$.

\begin{defi}
\label{FellerEvolution}
A family $\{P(s,t):0\leq s \leq t \leq T\}$ of operators defined on $L^{\infty}(E)$ is called a \textbf{Feller Evolution} or \textbf{Feller propagator} on $C_b(E)$ if it possesses the following properties: 
\begin{enumerate}
    \item It leaves $C_b(E)$ invariant: $P(s,t)C_b(E)\subset C_b(E)$ for $0\leq s\leq t \leq T$;
    \item It is an evolution: $P(\tau,t)=P(\tau,s)P(s,\tau)$ for all $\tau, s, t$ for which $0\leq \tau \leq s \leq t$ and $P(t,t)=I$, $t\in [0,T]$;
    \item If $0\leq f \leq 1$, $f\in C_b(E)$, then $0\leq P(s,t)f \leq 1$, for $0\leq s \leq t\leq T$. 
    \item If the function $(s,t,x)\mapsto P(s,t)f(x)$ is continuous on the space $\Lambda:=\{(s,t,x)\in [0,T]\times[0,T]\times E:s\leq t\}$. 
\end{enumerate}
\end{defi}
In the following theorem we see that with a Feller evolution a strong Markov process can be associated in such a way that the one-dimensional distributions or marginals are determined by the operators $f\mapsto P(\tau,t)f$, $f\in C_b(E)$. In fact every operator $P(\tau,t)$ can be written as 
\[
P(\tau,t)f(x)=\int P(\tau,x;t,dy)f(y)
\]
for $f\in C_b(E)$. Where the mapping 
\[
(\tau,x,t,B)\mapsto P(\tau,x;t,B)
\]
with $(\tau,x,t,B)\in [0,T]\times E\times [0,T]\times \mathcal{E}$, $\tau \leq t$, is a sub-probability transition function. 

\begin{teo}
\label{nAHunt}
Let $\{P(\tau,t):\tau\leq t \leq T\}$ be a Feller evolution in $C_b(E)$. Then there exists a strong Markov process (in fact a Hunt Process) such that $[P(\tau,t)f](X)=\mathbb{E}_{\tau,x}[f(X(t))]$, $f\in C_b(E)$, $t\geq0$. Moreover, this Markov process possesses the properties in \cite[Thm.\ 2.9.]{vanCasteren2011}. 
\end{teo}

\begin{proof}
\cite[Thm.\ 2.9]{vanCasteren2011}.
\end{proof}

\subsection{Stochastic process attached to a non-autonomous equation}

In this section we prove that the system 
\begin{equation}\label{nAsyst}
    \begin{cases}
      \frac{\partial u}{\partial t} (x,t)=\mathbb{L}(t)u(x,t)\\
      u(x,s)=u_s(x)\in C(K_n,\mathbb{C})
    \end{cases}
\end{equation}
generates a time-inhomogeneous Markov process, where the operators $\mathbb{L}(t)$ satisfies the hypothesis of Theorem \ref{stronglyContinuous}. In more generality, we study the following Cauchy problem 
\begin{equation}\label{nACP_stoch}
    \begin{cases}
      \frac{\partial u}{\partial t} (x,t)=A(t)u(x,t)\\
      u(x,s)=u_s(x)\in C(E)\,,
    \end{cases}
\end{equation}
where $E$ is a compact Polish space. Thus, $C_b(E)=C(E)$. Also, we suppose that $t\mapsto A(t)$ is continuous in the uniform norm topology (In particular $A(t)$ is bounded for all $t\geq 0$). Moreover, we suppose that for every $\tau\geq 0$, $A(\tau)$ generates a strongly continuous, positive, contraction semigroup $\{e^{tA(\tau)}\}_{t \geq 0}$ on $C(E)$. By Theorem \ref{fattoriniTheorem},
the system (\ref{nACP_stoch}) 
is well-posed for every initial condition $u_s(x)\in C(E)$. Therefore, there exists an evolution family $U(t,s)$, such that $U(t,s)u_s(x)$ is the unique solution of the Cauchy problem (\ref{nACP_stoch}). 
Moreover, $U(t,s)$ satisfies the properties in Remark \ref{solutionProperties}.

\begin{lem}\label{Feller}
The evolution $U(t,s)$ described above is a Feller evolution.
\end{lem}

\textit{Proof.} 
We verify the properties of Definition \ref{FellerEvolution}.
The verification of \emph{1}.\ and \emph{2}.\ is straightforward.
We now prove condition \emph{3}. Let $f\in C(E)$ such that $0\leq f \leq 1$. By Proposition \ref{uniformConvergence}, 
we have 
\[
\lim_{n\rightarrow \infty}\prod_{k=1}^{n}e^{\frac{(t-s)}{n} A(s+k\frac{(t-s)}{n})}f(x)=U(t,s)f(x)
\]
in the supremum norm, and with $0\leq s \leq t \leq T$. Therefore, we also have point-wise convergence. Since every semigroup $\{e^{tA(\tau)}\}_{t\geq 0}$ is a strongly continuous, positive, contraction semigroup we have 
\[
0\leq \prod_{k=1}^n e^{\frac{(t-s)}{n} A(s+k\frac{(t-s)}{n})}f(x)\leq 1.
\]
We obtain the desire result by taking $n\rightarrow \infty$. For condition \emph{4}., note that, since the space $E$ is Polish, we can prove the continuity in terms of sequences. Let $(t,s,x)\in \Lambda$, with $s\leq t$ and $x\in E$. And let $(t_n,s_n,x_n)$ a convergent sequence, with $s_n\leq t_n$ and $x_n\in E$, such that $(t_n,s_n,x_n)\rightarrow(t,s,x)$ as $n\rightarrow \infty$. We have

\begin{equation*}
    \begin{split}
      |U(t_n,s_n)f(x_n)-U(t,s)f(x)|&\leq |U(t_n,s_n)f(x_n)-U(t,s)f(x_n)|\\&+|U(t,s)f(x_n)-U(t,s)f(x) |  \\
      \leq &||U(t_n,s_n)-U(t,s)||_{\infty}||f||_{\infty}\\ &+|U(t,s)f(x_n)-U(t,s)f(x) |
    \end{split}
\end{equation*}
Since $(t,s)\mapsto U(t,s)$ is continuous in the uniform norm topology (Remark 
\ref{solutionProperties}.\emph{2}.), 
and the function $x\mapsto U(t,s)f(x)$ is continuous in $E$, by taking $n\rightarrow \infty$ the right hand side of the above inequality tends to zero. This now proves that $U(t,s)$ is a Feller evolution. \qed
\begin{teo}\label{contractionSG}
Let $\{A(t)\}_{t\geq 0}$ be a set of bounded operators on $C(E)$. Suppose that $t\mapsto A(t)$ is continuous in the uniform norm topology and that each $A(t)$ generates a strongly continuous, positive, contraction semi-group. Then the Cauchy problem (\ref{nACP_stoch})
is well-posed and its evolution family $U(t,s)$ generates a strong Markov process with the properties listed in Theorem \ref{nAHunt}. 
\end{teo}

\textit{Proof.} This is a consequence of Lemma \ref{Feller}
and Theorem \ref{nAHunt}. 
\qed

\begin{rem}
The literature does not seem to contain a result proving that a non-autonomous Cauchy problem has attached to it a Markov process. That is why  the above result was obtained using the techniques learned above in order to find an aproximate solution (Proposition \ref{uniformConvergence}). 
In particular, it is important to note that evolution families are the two-parameter generalizations of semigroups, and analogously, Feller evolution families are the two parameter generalizations of Feller semigroups. 
\end{rem}

\begin{teo}\label{Markov}
There exists a probabilty transition function $P(t,x;s,\cdot)$, where $(t,x,s)\in [0,T]\times K_N \times [0,T] $, and $s\leq t$, on the Borel $\sigma$-algebra of $K_N$, such that the Cauchy problem (\ref{nAsyst}) 
has a unique solution of the form 
\[
u(x,t)=\int_{K_N} u_s(x) P(t,x;s,dy).
\]
In addition, $P(t,x;s,\cdot)$ is the transition function of a Markov process with the properties listed in Theorem \ref{nAHunt}. 
\end{teo}
\textit{Proof.} We have to proove that $t\mapsto \mathbb{L}(t)$ satisfies the hypothesis in Theorem \ref{contractionSG}. 
By   \cite[Lem.\ 4.1]{ZunigaNetworks}, we know that $\mathbb{L}(t)$ generates a strongly continuous, positive, contraction semigroup on $C(K_N)$ for each $t\ge0$. hence, we only have to show that $t\mapsto \mathbb{L}(t)$ is continuous in the uniform norm topology. Recall from the proof of Theorem
\ref{stronglyContinuous} that 
\[
||\mathbb{L}(t)u-\mathbb{L}(s)u||_{\infty}\leq ||\mathcal{A}(t)u-\mathcal{A}(s)u||_{\infty}+||\Gamma(t)u-\Gamma(s)u||_{\infty}.
\]
Let $\varepsilon>0$, and $\delta>0$ such that 
\[
|t-s|<\delta \implies|A(x,y,t)-A(x,y,s)|<\varepsilon \ and  \ |\gamma(t)(x)-\gamma(s)(x)|< \varepsilon
\]
We have that
\begin{equation}
    \begin{split}
        ||\mathcal{A}(t)u-\mathcal{A}(s)u||_{\infty}&=\sup_{x\in K_N}\left|\int_{K_N} (A(x,y,t)-A(x,y,s))u(y)dy\right|\\ 
        &\leq \varepsilon||u||_{\infty}\Vol(K_N)
    \end{split}
\end{equation}
and 
\[
||\Gamma(t)u-\Gamma(s)u||_{\infty} =\sup_{x\in K_N}|(\gamma(t)(x)-\gamma(s)(x))||u(x)|\leq \varepsilon ||u||_{\infty}\,.
\]
If $||u||_{\infty}\leq 1$, then
\[
||\mathbb{L}(t)u-\mathbb{L}(s)u||_{\infty}\leq \varepsilon(\Vol(K_N)+1).
\]
This inequality implies the continuity in the uniform norm topology. We now apply Theorem \ref{nAHunt} and Theorem \ref{contractionSG} to obtain the result. \qed

\section*{Acknowledgements}

The first named author is indebted to Norbert Paul for coming up  with questions on how to study topological questions in the setting of finite spaces a long time ago, which gave a long-lasting impulse to eventually combine these with ultrametric analysis and mathematical physics.
The second named author is indebted to
Wilson Z\'u\~niga-Galindo for sharing valuable insights into $p$-adic analysis when he was under his supervision as an undergraduate student. Martin Breunig, Markus Jahn and David Weisbart are warmly thanked for fruitful discussions.
This work is supported by the Deutsche Forschungsgemeinschaft
under project number 469999674.

\bibliographystyle{plain}
\bibliography{biblio}

\end{document}